\documentclass[11pt,
]{article}

\usepackage{graphicx,amsmath,amsthm,amsfonts,amscd,color}

\newtheorem{thm}{Theorem}[section]
\newtheorem{prop}[thm]{Proposition}
\newtheorem{lem}[thm]{Lemma}

\newcommand{\CP}[1]{\ensuremath{\mathbf{CP}^{#1}}}
\newcommand{\C}{\ensuremath{\mathbf{C}}}

\newcommand{\CC}[1]{\ensuremath{\mathcal{C}_{#1}}}

\newcommand{\T}[1]{\ensuremath{\mathcal{T}_{#1}}}
\newcommand{\sym}[1]{\ensuremath{\mathcal{S}_{#1}}}
\newcommand{\alt}[1]{\ensuremath{\mathcal{A}_{#1}}}
\newcommand{\I}{\ensuremath{\mathcal{I}}}

\title{Completely solving the quintic by iteration}
\author{Scott Crass\\
Department of Mathematics and Statistics\\
California State University, Long Beach\\
Long Beach, CA  90840\\
scrass@csulb.edu}

\begin{document}

\maketitle

\begin{abstract}
	
In the late nineteenth century, Felix Klein revived the problem of solving the quintic equation from the moribund state into which Galois had placed it.  Klein's approach was a mix of algebra and geometry built on the structure of the regular icosahedron.  His method's key feature is the connection between the quintic's Galois group and the rotational symmetries of the icosahedron.

Roughly a century after Klein's work, P.\ Doyle and C.\ McMullen developed an algorithm for solving the quintic that also exploited icosahedral symmetry.  Their innovation was to employ a symmetrical dynamical system in one complex variable.  In effect, the dynamical behavior provides for a partial breaking of the polynomial's symmetry and the extraction of two roots following one iterative run of the map.  

The recent discovery of a map whose dynamics breaks \emph{all} of the quintic's symmetry allows for all five roots to emerge from a single run.  After sketching some algebraic and geometric background, the discussion works out an explicit procedure for solving the quintic in a complete sense.

\end{abstract}

\section{Overview}

Solving a polynomial equation calls for a means to overcome the polynomial's symmetry.  In the case of the fifth-degree equation, the general symmetry group is the symmetric group \sym{5}.  In terms of Galois theory, we can reduce the symmetry to that of the alternating group \alt{5} by adjoining the square root of the polynomial's discriminant to the coefficient field.  Our reward for this reduction is that we can realize \alt{5} as the rotational symmetries of the regular icosahedral configuration of the complex projective line \CP{1}---that is, the Riemann sphere.

By exploiting icosahedral structure, Doyle and McMullen constructed a quintic-solving algorithm at the core of which is a map $\phi$ that respects the \alt{5} symmetry.\cite{dm}  The map is strongly critically finite, meaning that it's critical set \CC{\phi}, the twenty face-centers of the icosahedron, is $\phi$-invariant; that is, $\phi(\CC{\phi})=\CC{\phi}$.  In particular, each superattracting critical point has period two.  It follows that almost every point in \CP{1} belongs to the basin of attraction of some two-cycle in \CC{\phi}.  The procedure employs $\phi$'s dynamics in a way that partially breaks the \alt{5} symmetry and, with one iterative run, computes two roots.

Recent work determined all icosahedrally-symmetric maps with \emph{internally periodic} critical sets of size $60$.\cite{soccer,icos31}  Internal periodicity means that the map acts on its critical set as a permutation.  Here, we build a quintic-solving device around the dynamics of one such map $g$ whose critical points have period five.  Since the superattracting set has generic size, the dynamics of $g$ effectively breaks \emph{all} of an equation's \alt{5} symmetry.  Accordingly, the algorithm produces all five roots with a single iterative run.

Computational results are produced by \emph{Mathematica} and basins-of-attraction plots are the product of \emph{Dynamics 2}. \cite{dyn2} 
\section{Icosahedral algebra: invariants, and equivariants}

An account of the algebraic objects that emerge from the icosahedral action on \CP{1} appears in other places (\cite{klein}, \cite{dm}, and \cite{soccer}).  Results relevant to the task at hand appear without discussion.

Denote by \I\ the \alt{5}-isomorphic group of $60$ rotational symmetries of the regular icosahedron as a graph structure on the sphere.  (See Figure~\ref{fig:5tet}.)  Three polynomials generate the ring of \I-invariants:
$$\C[x,y]^\I=\left<F(x,y), H(x,y), T(x,y)\right>$$
where $(x,y)$ are homogeneous coordinates on \CP{1}.  The forms $F$, $H$, and $T$ vanish at the special \I-orbits: the $12$ vertices $v_k$, $20$ face-centers $f_k$, and $30$ edge-midpoints $e_k$ respectively.  For ease of reference, call the members of these sets ``$12$-points," etc.  We can express the generating invariants as products:
\begin{align*}
F=&\ \prod_{k=1}^{12} (x-v_k y)= x y \left(x^{10}-11x^5 y^5-y^{10}\right)\\
H=&\ \prod_{k=1}^{20} (x-f_k y)= x^{20}+228x^{15} y^5+494x^{10} y^{10}-228x^5
y^{15}+y^{20} \\
T=&\ \prod_{k=1}^{30} (x-e_k y)=  x^{30}-522x^{25} y^5-10005x^{20} y^{10}-10005x^{10} y^{20}+522x^5 y^{25}+y^{30}.
\end{align*}

Accordingly, $F$ and $H$  are algebraically independent, while we can arrange for an algebraic combination of the two generators in degree $60$ to vanish (with multiplicity two) at the $30$-points:
$$T^2=H^3-1728\,F^5.$$

We also need the system of invariants for each of the five tetrahedral subgroups $\T{1},\dots,\T{5}$ of \I.  Each \T{k} acts as an alternating group \alt{4} on a set of four $20$-points.  Overall, these disjoint sets occupy the vertices of five regular tetrahedra. Figure~\ref{fig:5tet} shows the icosahedral net and the vertical decomposition into tetrahedral sets.  Taking $k=5$, there are two degree-four relative \T{5} invariants: one, $q_5$, given by the product that involves the tetrahedral vertices and the other, $\Hat{q}_5$, given by the product that uses the tetrahedral face-centers (antipodal to the vertices).  The results are
\begin{align*}
q_5=&\  \frac{1}{4} \bigl(4 x^4+(2+2 i \sqrt{15}) x^3 y+(6-2 i \sqrt{15}) x^2 y^2+(-2-2 i \sqrt{15}) x y^3+4 y^4\bigr)\\
\Hat{q}_5=&\ \frac{1}{2} \bigl(2 x^4+ (1-i \sqrt{15})x^3 y)+(3+i \sqrt{15}) x^2 y^2-
(1-i\sqrt{15}) x y^3+2 y^4\bigr).
\end{align*}
\begin{figure}[ht]
	\centering
	
	\includegraphics[width=4in]{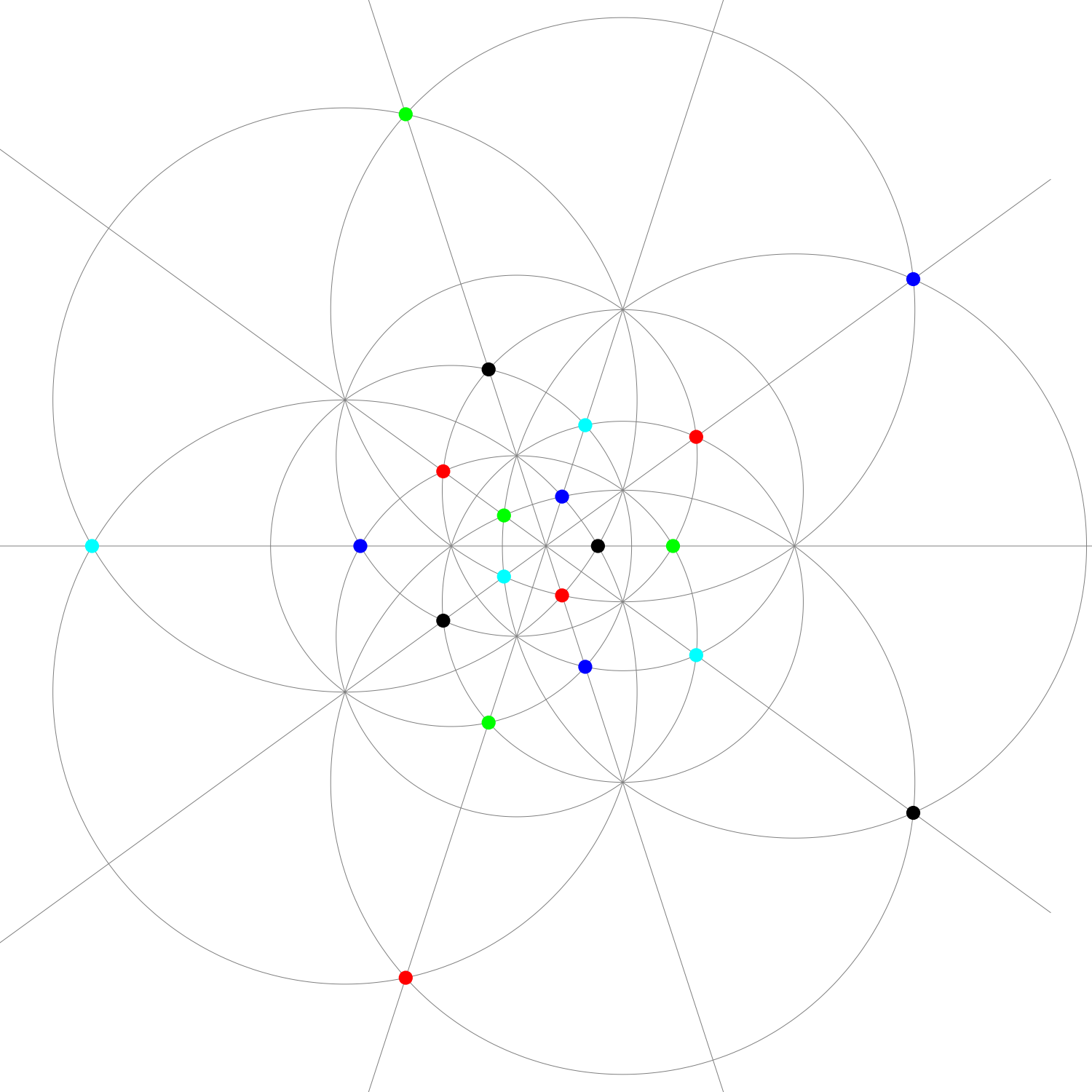}
	\caption{Configuration of five tetrahedral vertices}
	
	\label{fig:5tet}
	
\end{figure}
A \emph{relative invariant} is a form for which a non-trivial multiplicative character appears under the group's action.  For instance, for some $A\in \mathcal{T}_5$,
$$q_5(A(x,y))=\lambda_A\, q_5(x,y)\qquad \lambda_A\neq 1.$$

Using in a product the six-point tetrahedral orbit associated with edges  gives an absolute \T{5} invariant:
$$t_5= x^6-2 x^5 y-5 x^4 y^2-5 x^2 y^4+2 x y^5+y^6.$$
In this case, $t_5(A(x,y))=t_5(x,y)$ for all $A\in \mathcal{T}_5$.  
The product of the degree-four forms yields a degree-eight \T{5} invariant:
$$
u_5=q_5 \widehat{q}_5=x^8+x^7 y+7 x^6 y^2-7 x^5 y^3+7 x^3 y^5+7 x^2 y^6-x y^7+y^8.
$$
Since the eight zeroes of $u_5$ have order-three symmetry, they are also face-centers of the icosahedron.  Hence, $H$ is divisible by $u_5$ and the quotient is a \T{5} invariant of degree $12=20-8$:
\begin{align*}
m_5=\frac{H}{u_5}=&\ x^{12}-x^{11} y-6 x^{10} y^2+20 x^9 y^3+15 x^8
y^4+24 x^7 y^5+11 x^6 y^6-24 x^5 y^7+15 x^4
y^8\\
&\ -20 x^3 y^9-6 x^2 y^{10}+x y^{11}+y^{12}.
\end{align*}
The tetrahedral invariants satisfy a  degree-$24$ relation
$$
m_5^2=\frac{1}{64} \left(95 t_5^2 m_5-40 t_5^4+9 u_5^3\right).
$$
Applying powers of an order-$5$ element $P\in \I$ manufactures the remaining \T{k} invariants:
$$
t_k=t_5\circ P^k  \qquad u_k=u_5\circ P^k \qquad m_k=m_5\circ P^k \qquad  k=1,\dots,4.
$$
In the chosen coordinates, we can take $P(x,y)=(\epsilon^3 x,\epsilon^2 y)$ where $\epsilon=e^{2 \pi i/5}$.

From a generating \I-invariant, we can construct an \I-equivariant (or \I-map) of one less degree using a ``cross" operator $\times$:
\begin{align*}
\phi=&\ \times F=(-\partial_y F,\partial_x F)= \bigl(-x^{11}+66 x^6 y^5+11 x y^{10},11 x^{10} y-66 x^5 y^6-y^{11}\bigr)\\
\eta=&\ \times H=(-\partial_y H,\partial_x H)\\
=&\ 20 \bigl(-57 x^{15} y^4-247 x^{10} y^9+171 x^5 y^{14}-y^{19} ,
x^{19}+171 x^{14} y^5+247 x^9 y^{10}-57 x^4 y^{15}\bigr).
\end{align*}
These maps behave in an elegant manner: $\phi$ twists and wraps a dodecahedral  face $\mathcal{F}$ onto the $11$ faces that comprise the complement of the face antipodal to $\mathcal{F}$ while $\eta$ does the analogous twisting and wrapping for an icosahedral face.  For edges of the respective polyhedra we can take great circle arcs between vertices to obtain sets that are forward invariant under the respective map.  Call this structure a \emph{dynamical polyhedron}.  Moreover, each map expands the internal angle of a face in its dynamical polyhedron onto an external angle of the antipodal face.  The vertices are thereby periodic critical points and their superattracting basins are full-measure subsets of \CP{1}.  The Doyle-McMullen iteration uses $\phi$ whose  attracting set is a special orbit.  Hence, \alt{5} symmetry is partially broken allowing for the extraction of two roots.

To break \alt{5} symmetry fully, we look for a map $g$ whose critical set \CC{g} is a generic $60$-point \I-orbit that is permuted under the action of $g$.  All maps of this sort have degree $31$ and were found in \cite{icos31}.  Excepting two cases, the dynamical polyhedra associated with these special ``$31$-maps" are derived from the icosahedral structure; they consist of twelve pentagons, twenty triangles, and thirty quadrilaterals.  The resulting configuration is called a $B_{62}$.  (It also goes by the awkward name rhombicosidodecahedron.).

\section{A special map}

Take for $g$ a map with period-five critical points so that each five-cycle resides at the consecutive pentagonal vertices on the $B_{62}$.  (In \cite{icos31}, I describe how this map was found as well as the other critically-finite degree-$31$ \I-maps.)   Its analytic form is approximated by
\small
\begin{align*} g=\ & \alpha H\cdot \phi\ +\ \beta F\cdot \eta  \\
\approx\ & \bigl(-19 x (x^{30}-(487.5215055+65.4865970 i) x^{25} y^5 - (10234.856630-436.577313 i)x^{20} y^{10}\\
&-(1781.388882-3383.474177 i)x^{15} y^{15}-
(9016.011606+1878.431334 i)x^{10} y^{20}\\
&+(618.7817389-183.8220266 i) x^5y^{25} + 
(0.3951141318+1.1488876663 i)y^{30},\\
&(-7.50716850-21.82886566 i)(x^{30} y-(22.5591063-530.8337230 i)
x^{25} y^6\\
&-(3875.498123-6514.777510 i) x^{20}
y^{11}-(2156.676586+2292.236531 i) x^{15}y^{16}-\\
&(2399.877301-8083.149873 i) x^{10}y^{21}+
(181.4721179-361.9320839 i) x^5y^{26}\\
&+(0.2676819741-0.7783485674 i)y^{31})\bigr)
\end{align*}
\normalsize
where $(\alpha,\beta)\approx(19,-10.825358425-1.091443283 i)$.

As discussed in \cite{soccer} and \cite{icos31}, $g$'s geometric behavior gives rise to a polyhedral system of ``edges" $\mathcal{E}_g$ that forms a forward invariant set.  This collection of edges fills in the $B_{62}$ structure whose faces consist of twelve pentagons, twenty triangles, and thirty quadrilaterals that realize five-fold, three-fold, and two-fold rotational symmetry respectively.  Figure~\ref{fig:dyn_poly} shows the output of an algorithm worked out in \cite{icos31} that constructs an approximation to the edge-system overlaid on a coloring scheme determined by the map's topological behavior.  
\begin{figure}[ht] \centering
	\includegraphics[width=2.7in]{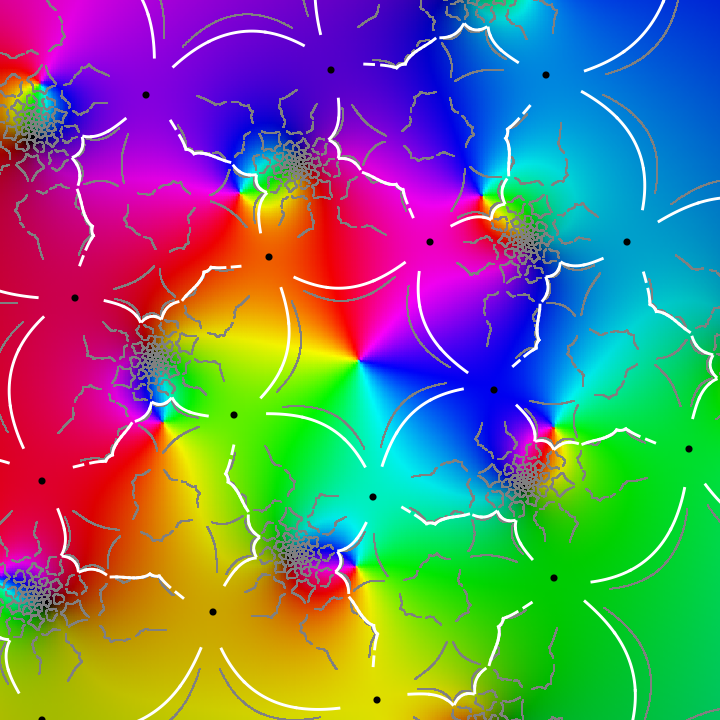}
	\includegraphics[width=2.7in]{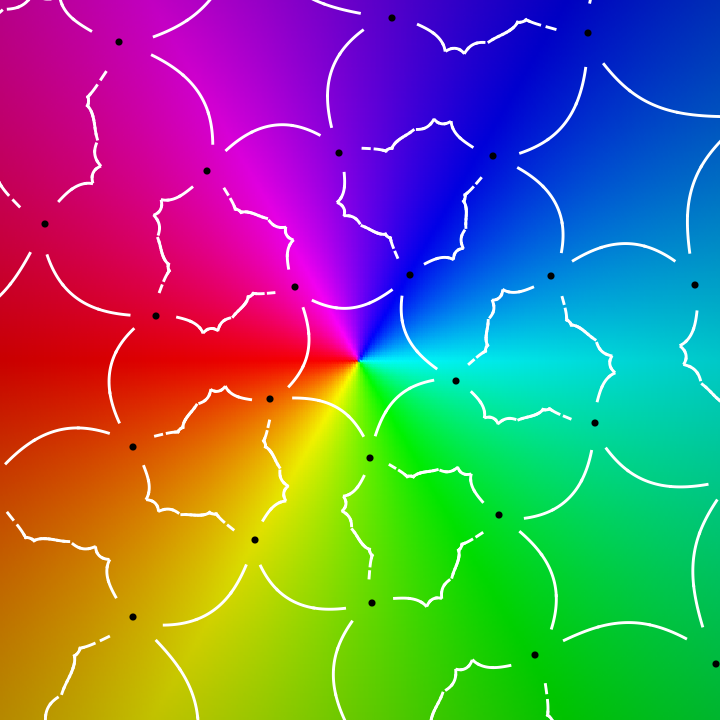}
	\caption{On the right we see an approximation of a $B_{62}$ in \C{}.  The plot exhibits a color-luminosity $(C,L)$ field in which each point $z$ receives a $(C(z),L(z))$ coordinate determined by $(\arg(z),\lvert z \rvert)$.  The plot on the left reveals the combinatorial behavior of $g$ in which a point $z$ is colored as $(C(g(z)),L(g(z))).$  So, the map's behavior is evident by matching color-luminosity values between left and right fields.  The overlaid curves (gray) in the left plot outline the regions---call them ``pre-faces"---that map to the respective types of face (outlined in white) on the right.  The algorithm that generates the edges relies on backward iteration; hence, the appearance of gaps around the critical points.   Specifically, the image of a pentagon covers a pentagon, a triangle covers three pentagons, one triangle, and four quadrilaterals, while a quadrilateral covers ten pentagons, $20$ triangles, and $29$ quadrilaterals.
}.
	\label{fig:dyn_poly}
\end{figure}
In Figure~\ref{fig:basin}, basin-of-attraction plots reveal $g$'s symmetry and global dynamics.
\begin{figure}[ht] \centering
\includegraphics[width=2.7in]{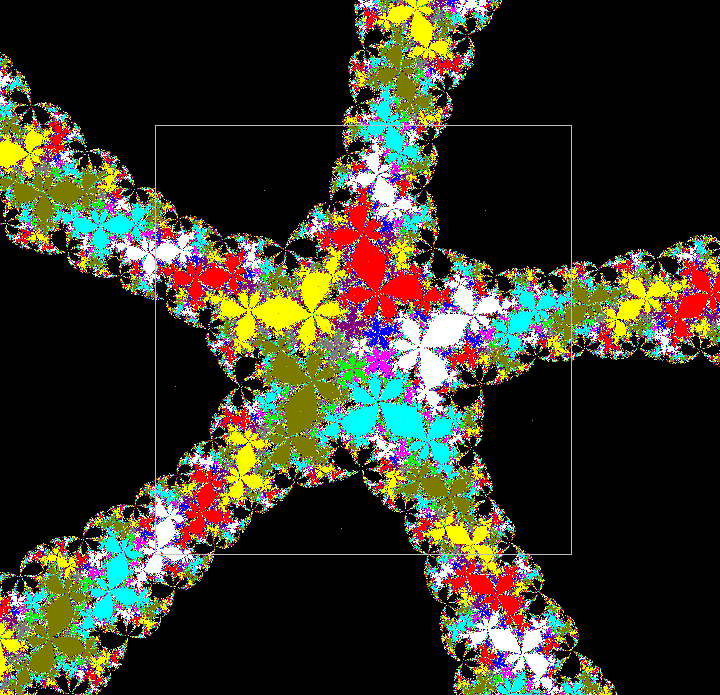}
\includegraphics[width=2.7in]{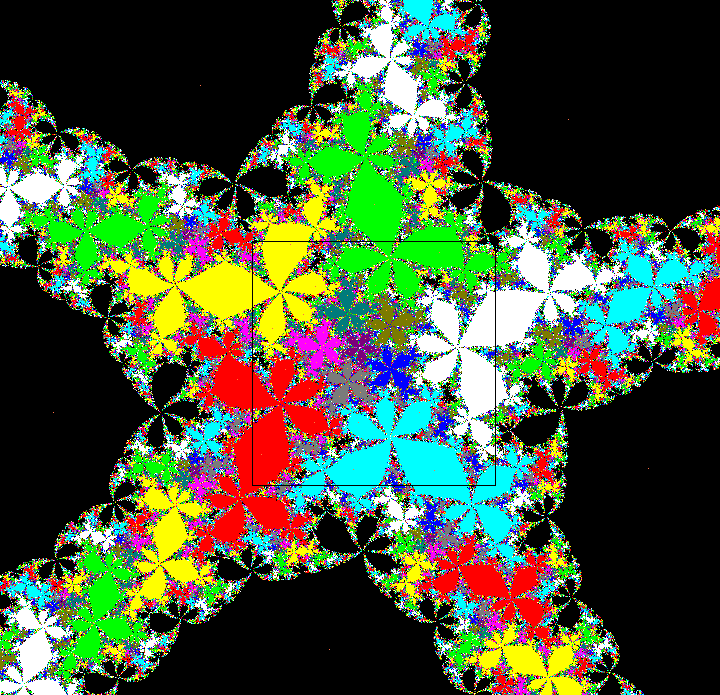}\\[5pt]
\includegraphics[width=2.7in]{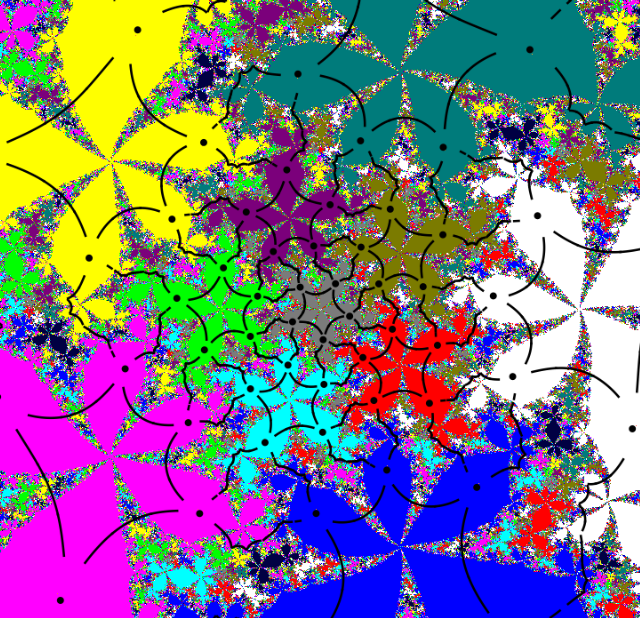}
\includegraphics[width=2.7in]{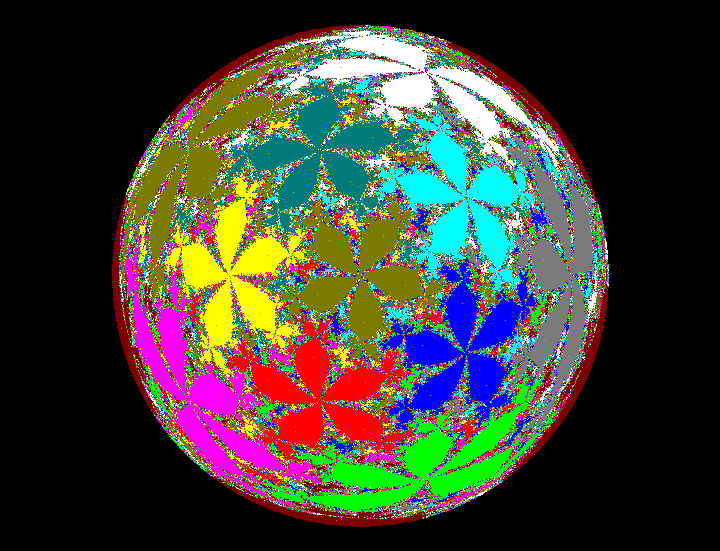}
\caption{Basins of attraction for twelve critical five-cycles.  The top right plot is a magnified view of the region in the square box shown in the top left plot while the bottom left image displays the region bounded by the square in the plot at top right and also displays the edge-system $\mathcal{E}_g$.  The view at bottom right shows the basins under a projection of the plane onto a disk.  Note that the coloring is inconsistent between the plots.}
\label{fig:basin}
\end{figure}

By critical-finiteness, the orbit of almost every $p\in \CP{1}$ tends to a critical five-cycle:
$$
g^k(p)\overset{k\rightarrow \infty}{\longrightarrow} (r_1,r_2,r_3,r_4,r_5)\subset \CC{g}.
$$
By the cycling action of $g$ on the adjacent pentagonal vertices $(r_1,\dots,r_5)$, we can take each $r_k$ to be a vertex of the tetrahedron invariant under $\mathcal{T}_k$---evident in Figure~\ref{fig:5tet}.  This dynamical outcome lies at the core of a quintic-solving procedure and the presence of a period-five attracting set makes for an elegant algorithm. 

\clearpage
\section{Solving the quintic}

\subsection{Resolvent}

First, we create a parametrized family of quintic equations that our dynamical algorithm will solve. Let 
$$\rho_k(x)=\frac{F(x)u_k(x) }{H(x)}$$
where $x=(x_1,x_2)$ are homogeneous coordinates replacing the former $(x,y)$.  Take the degree-zero rational functions $\rho_k$ as five roots of an polynomial:
$$R_x(v)=\prod_{k=1}^5 (v-\rho_k(x))=\sum_{j=0}^5  b_j(x) v^j.$$
By construction, the tetrahedral forms $u_k$---hence, the $\rho_k$---are permuted by \I.  Accordingly, the coefficients $b_j$, being symmetric functions in the $\rho_k$, are \I-invariant and, thereby expressible in terms of $F$ and $H$.  

Note that some of the coefficients vanish due to their degree.  For instance, the coefficient of $v^4$ is
$$b_4=\frac{F \sum_{k=1}^5 u_k}{H}.$$
Since $\sum_{k=1}^5 u_k$ is degree-eight and there are no such \I-invariants, it turns out that $b_4=0$.  As for the surviving coefficients, each $b_k$ is rational of degree-zero, and so, admits expression in the icosahedral parameter $Z=\frac{F^5}{H^3}$.  The result is a one-parameter family of quintic resolvents
$$R_Z(v)=v^5-40 Z v^2-5 Z v-Z.$$
For the sake of completeness, note that L.\ Dickson reduced the general quintic to a one-parameter resolvent.\cite[Ch.\ XIII]{dickson}

In the construction of a quintic-solving algorithm, the key step occurs when a quintic $R_Z$ is connected with a map $g_Z$ each of which is conjugate to the special \I-map $g$.  Finally, we'll build a function---also parametrized by $Z$---that will convert $g_Z$'s dynamical output into the roots of a chosen $R_Z$.  

\subsection{Parametrization}

To begin the parametrization process, consider the family of transformations 
$$x=S_y{w}=H(y) \phi(y) w_1 + F(y) \eta(y) w_2$$
that is linear in $w=(w_1,w_2)$ and degree-$31$ in $y=(y_1,y_2)$.  The coordinate $y$ substitutes identically for $x$ with its associated icosahedral group $\I_y$ and serves as a parameter.  Accordingly, the transformation enjoys an equivariance property:
$$S_{Ay}=AS_y \qquad \text{for all $A\in \I_y$}.$$
Figure~\ref{fig:param_coord} shows each $S_y$ as a coordinate change from the $y$-parametrized $w$-space and icosahedral action $\I_y^w$ to the fixed $x$-space with action $\I^x$.
\begin{figure}[ht]
$$\begin{CD}
\CP{1}_x       @>\I^x>> \CP{1}_x &\quad \text{(reference space)}\\
@A{S_y}AA        @VV{S_y^{-1}}V\\
\CP{1}_w   @>\I_y^w>> \CP{1}_w &\quad \text{(parametrized space)}\\
@A{S}AA        @AA{S}A\\
\CP{1}_y   @>\I_y>> \CP{1}_y  &\quad \text{(parameter space)}
\end{CD}$$
\caption{Parametrizing the icosahedral action}
\label{fig:param_coord}
\end{figure}

With coordinate transformation $S_y$ in hand, we can construct the generating invariants and equivariants under $\I_y^w$.  Taking the degree-$12$ invariant
$$F(x)=F(S_y{w})=\sum_{k=0}^{12} a_k(y) w_1^{12-k}w_2^k,$$
the result is a polynomial whose $w$-degree is $12$ while each $a_k(y)$ has a $y$-degree of $12\cdot 31$.  Moreover, each $a_k(y)$ is invariant under $\I_y$ and thereby expressible as a polynomial $\Hat{a}_k(F(y),H(y))$.  Hence, we get
$$\Hat{F}(F(y),H(y)):=F(S_y{w})=\sum_{k=0}^{12} \Hat{a}_k(F(y),H(y)) w_1^{12-k}w_2^k.$$
Note that, by degree considerations, the degree-$30$ form $T(y)$ cannot appear to an odd power in the invariant expression for $a_k(y)$ whereas $T(y)$ raised to an even power converts to a polynomial in $F(y)$ and $H(y)$.  Dividing by $F(y)^{31}$ ``normalizes" $F(x)$ to a degree-zero rational function in $y$ from which we obtain a $Z$-parametrized function: 
\small
\begin{align*}
F_Z(w)=& \frac{F(x)}{F(y)^{31}}= \frac{F(S_y{w})}{F(y)^{ 31}}\biggr|_{H(y)^3\rightarrow F(y)^5 Z}\\
	=&\ Z^{-6}\bigl(
	4096000000000000 w_1^{12} Z^3 (16 Z (432 Z (432 Z-95)-437)+57)\\
	&-204800000000000 w_1^{11} w_2 Z^2 (132 Z (864 Z (216 Z+5)-47)-1)\\
	&-112640000000000 w_1^{10} w_2^2 Z^2 (8 Z (864 Z (4104 Z+245)-3443)-11)\\
	&-28160000000000 w_1^9 w_2^3 Z^2 (32 Z (216 Z (3456 Z+833)-4961)-121)\\
	&-4224000000000	w_1^8 w_2^4 Z^2 (864 Z (20952 Z-1147)-1331)\\
	&-337920000000 w_1^7 w_2^5 Z^2 (432 Z (131328 Z-18053)-18287)\\
	&-704000000 w_1^6 w_2^6 Z (48 Z (432 Z (138240 Z-76183)-140479)+1)\\
	&+211200000 w_1^5 w_2^7 Z (432 Z (3314304 Z+28501)-11)\\
	&+26400000 w_1^4 w_2^8 Z (432 Z (4202496 Z+89177)-121)\\
	&+1760000 w_1^3 w_2^9 Z (13824 Z (138240 Z+11477)-1331)\\
	&+8553600 w_1^2	w_2^{10} Z (6027264 Z-113)\\
	&+w_1 w_2^{11} (69120 Z (84049920 Z-3077)-20)\\
	&+w_2^{12} (1769472 Z (172800 Z-11)-11)
	\bigr).
\end{align*}
\normalsize
To convey a sense of the result, the full expression is quoted here.  The lengthy formulas for subsequent computations will be suppressed and can be found at \cite{web}.
Applying the same technique generates a function
$$
H_Z(w)= \frac{H(x)}{H(y)^{31}}= \frac{H(S_y{w})}{H(y)^{31}}\biggr|_{H(y)^3\rightarrow F(y)^5 Z}
$$
whose $w$-degree is $20$.  Similarly for maps:
$$
\phi_Z(w)= \frac{\phi(x)}{F(y)^{31}}= \frac{\phi(S_y{w})}{F(y)^{31}}\biggr|_{H(y)^3\rightarrow F(y)^5 Z}\qquad
\eta_Z(w)= \frac{\eta(x)}{H(y)^{31}}= \frac{\eta(S_y{w})}{H(y)^{31}}\biggr|_{H(y)^3\rightarrow F(y)^5 Z}.
$$

Next, we develop a $Z$-parametrized version of $g$ defined on $\CP{1}_w$ the first step of which is to express the cross operator as
$$
\times_x P(x) = J \nabla_x P(x)\qquad \text {where}\
J=\begin{pmatrix}0&-1\\1&0\end{pmatrix}\ \text{and}\ 
\nabla_x=\biggl(\frac{\partial}{\partial x_1},\frac{\partial}{\partial x_2}\biggr).
$$
Straightforward computations capture how the operator transforms under a linear change of coordinates $x=Aw$ on $\C^2$.  
\begin{lem}
$J (A^T)^{-1} = A J\ \text{where $A^T$ is the transpose}$.
\end{lem}
\begin{prop}\label{prop}  Let $\lvert A \rvert$ denote the determinant and take operator subscripts to specify differentiation variables.  Then
\begin{align*}
\times_x P(x)=&\ J \nabla_x P(Aw)\\
=&\ J \vert A \rvert^{-1} (A^T)^{-1} \nabla_w P(Aw)\\
=&\ \delta J (A^T)^{-1} \nabla_w P(Aw)\qquad (\delta=\vert A \rvert^{-1})\\
=&\ \delta A\, J\, \nabla_w P(Aw)\\
=&\ \delta A (\times_w P(Aw)).
\end{align*}
\end{prop}
Regarding the constant $\delta$ as projectively meaningless, this transformation rule establishes a semi-conjugacy
$$
f(Aw)=\delta A \Hat{f}(w)\quad \text{with}\ f(x)= \times_x P(x)\ \text{and}\ \Hat{f}(w)=\times_w P(Aw).
$$
Applying the formula derived in Proposition~\ref{prop} to the basic icosahedral maps yields
\begin{align*}
\phi(x)=&\ \times_x F(x)\\
=&\ \times_x F(S_y w)\\
=&\ \lvert S_y \rvert^{-1} S_y (\times_w F(y)^{31} F_Z(w))\\
=&\ F(y)^{31} \lvert S_y \rvert^{-1} S_y (\times_w F_Z(w))\\
=&\ F(y)^{31} \lvert S_y \rvert^{-1} S_y (\phi_Z(w))
\end{align*}
and
\begin{align*}
\eta(x)=&\ \times_x H(x)\\
=&\ \times_x H(S_y w)\\
=&\ \lvert S_y \rvert^{-1} S_y (\times_w H(y)^{31} H_Z(w))\\
=&\ H(y)^{31} \lvert S_y \rvert^{-1} S_y (\times_w H_Z(w))\\
=&\ H(y)^{31} \lvert S_y \rvert^{-1} S_y (\eta_Z(w)).
\end{align*}
With these transformation properties, we catch sight of a map on the $w$-space that is dynamically equivalent to $g(x)$.  With $\alpha$ and $\beta$ as determined previously,
\begin{align*}
g(x)=&\ \alpha\,F(x) \eta(x) + \beta\,H(x) \phi(x)\\
=&\ \alpha\,F(y)^{31} F_Z(w) H(y)^{31} \lvert S_y \rvert^{-1}  S_y (\eta_Z(w))\\
&+\ \beta\, H(y)^{31} H_Z(w) F(y)^{31}\lvert S_y \rvert^{-1}  S_y (\phi_Z(w))\\
g(S_y w)=&\ F(y)^{31} H^{31}(y) \lvert S_y \rvert^{-1} S_y
(\alpha\,F_Z(w)\eta_Z(w) + \beta\,H_Z(w) \phi_Z(w)).
\end{align*}
From this result, we take 
$$g_Z(w)=\alpha\,F_Z(w)\eta_Z(w) + \beta\,H_Z(w) \phi_Z(w)$$
to be projectively semi-conjugate to $g(x)$.  Accordingly, the $g_Z$-orbit of a random initial condition $w_0$ in $\CP{1}_w$ is asymptotic to a superattracting five-cycle
$$(\omega_1, \dots, \omega_5)=(S^{-1}r_1,\dots,S^{-1}r_5)$$
determined by $\I_y^w$.  Naturally, $(r_1,\dots,r_5)$ is a five-cycle of adjacent pentagonal vertices in \CC{g} under the action of $\I^x$ on $\CP{1}_x$.

\subsection{Root-selection}

The final step is to assemble an algorithm that uses the random nature of $w_0$ to effectively break \alt{5} symmetry entirely and obtain all of $R_Z$'s roots.  To that end, we'll fabricate a tool that will select the roots of the resolvent following a single iterative run of $g_Z$.  For each tetrahedral subgroup $\mathcal{T}_k$, consider the degree-$12$ family of invariants
$$h_k(x)=\gamma\,t_k(x)+\theta\,m_k(x).$$
Let \CC{k} be the twelve-element subset of \CC{g} that $\mathcal{T}_k$ preserves.  Tune one of the parameters $\gamma$ or $\theta$ so that
$$
h_k(p)=\begin{cases}0&p\in \CC{k}\\
c\neq 0&p\in \CC{g}-\CC{k}
\end{cases}
$$
and define the degree-$48$ form
$$\Tilde{h}_k(x)= \frac{\prod_{\ell=1}^5 h_\ell(x)}{h_k(x)}.$$
In practice, it's more convenient to calculate $\Tilde{h}_k$ by working with undetermined coefficients in the degree-$48$ family of tetrahedral invariants.
Now, spend the remaining parameter in order to normalize the degree-zero function
$$B_k(x)= \frac{\Tilde{h}_k(x)}{F(x)^4}$$
thereby obtaining specific behavior on the critical set of $g(x)$:
$$
B_k(p)=\begin{cases}1&p\in \CC{k}\\
0&p\in \CC{g}-\CC{k}
\end{cases}.
$$

Pairing the $y$-parametrized $B_k(S_y w)$ with the roots $\rho_k$ of the resolvent $R_Z$ leads to a \emph{root-extraction} function in the $Z$ parameter:
\begin{align*}
\Gamma_y(w)=&\ \sum B_k(x) \rho_k(y)=B_k(S_y w) \frac{u_k(y) F(y)}{H(y)}\\
=&\ \frac{F(y)}{F(y)^{124} F_Z(w)^4 H(y)} \sum \Tilde{h}_k(S_y w) u_k(y)\\
=&\ \frac{\sum \Tilde{h}_k(S_y w) u_k(y)}{F(y)^{123} H(y)} \frac{1}{F_Z(w)^4}.
\end{align*}
Here, we take a bare summation to run from 1 to 5.
\begin{prop}
The factor
$$\Lambda(y,w)=\sum \Tilde{h}_k(S_y w) u_k(y)$$
is $\I_y$ invariant.
\end{prop} 
\begin{proof}
Let $A\in \I_y$.  By the $\I_y$ equivariance of $S_y w$ as well as the congruent permutation action of $\I_y$ on $\Tilde{h}_k(S_y w)$ and $u_k(y)$,
$$
\Lambda(Ay,w)= \sum \Tilde{h}_k(S_{A y} w) u_k(A y)= \sum \Tilde{h}_k(A S_y w) u_{\sigma(k)}(y)
= \sum \Tilde{h}_{\sigma(k)}(S_y w) u_{\sigma(k)}(y)=\Lambda(y,w)
$$
where $\sigma$ is a permutation on $\{1,\dots,5\}$.
\end{proof}
By $\I_y$-invariance, we obtain a $Z$-parametrized function
$$L_Z(w)=\frac{\Lambda(y,w)}{F(y)^{123} H(y)}$$
and finally a root-extractor
$$\Gamma_Z(w)=\frac{L_Z(w)}{F_Z(w)^4}$$
associated with the resolvent $R_Z$.

To see how the extraction process works, fix a value for $y$ and let $q\in \CC{g_Z}\subset \CP{1}_w$ so that
$$q=S_y^{-1} p\quad \text{for some}\ p\in \CC{j}\subset \CP{1}_x.$$
Evaluating the selection function gives
\begin{align*}
\Gamma_y(q)=&\ \frac{L_Z(q)}{F_Z(q)^4}=
\frac{\sum \Tilde{h}_k(S_y S_y^{-1} p) u_k(y)}{F(y)^{123} H(y) F_Z(q)^4}\\
=&\ \frac{\sum \Tilde{h}_k(p)}{F(y)^{124} F_Z(q)^4} \frac{u_k(y) F(y)}{H(y)}\\
=&\ \frac{\sum \Tilde{h}_k(p)}{F(y)^{31\cdot4} F_Z(q)^4}  \rho_k(y)\\
=&\ \frac{\sum \Tilde{h}_k(p)}{(F(y)^{31} F_Z(q))^4}  \rho_k(y)\\
=&\ \sum \frac{\Tilde{h}_k(p)}{F(p)^4} \rho_k(y)\\
=&\ \sum B_k(p) \rho_k(y)\\
=&\ \rho_j.
\end{align*}

\subsection{Algorithm}
\begin{enumerate}
\item Select a random value $Z_0$ for the icosahedral parameter $Z=\frac{F^5}{H^3}$ and obtain a quintic resolvent 
$$R_{Z_0}=v^5-40 Z_0 v^2-5 Z_0 v-Z_0.$$
\paragraph{Remark.}  Solving $R_{Z_0}=0$ amounts to inverting the quotient map given by $Z(x)=Z_0$ in as much as the solutions form a single icosahedral orbit
$$\{x\in \CP{1}\ \vert\ F(x)^5-Z_0 H(x)^3=0\}.$$  That is, with the elements $x\in Z^{-1}(Z_0)$ we can produce the roots of $R_{Z_0}$ by evaluating $\rho_k(x)$ for $k=1,\dots,5$.  Such an inversion requires the complete breaking of $60$-fold \alt{5} symmetry, an outcome that's achieved dynamically.  Ultimately, this result amounts to the inversion of the elementary symmetric functions that make up the coefficients of the general quintic.
 
\item Compute the invariants $F_{Z_0}(w)$ and $H_{Z_0}(w)$ from which the equivariants $\phi_{Z_0}(w)$ and $\eta_{Z_0}(w)$ follow.  Determine the map
$$g_{Z_0}(w)= \alpha\,F_{Z_0}(w) \eta_{Z_0}(w) + \beta\,H_{Z_0}(w) \phi_{Z_0}(w)$$
on $\CP{1}_w$.
\paragraph{Remark.}  The parameter $Z$ is the harness that attaches $R_Z$ to $g_Z$.

\item Randomly select an initial condition $p\in \CP{1}_w$ and compute the orbit $\bigl(g_{Z_0}^k(p)\bigr)$ until it homes in on values $(\tilde{\omega}_1,\dots,\tilde{\omega}_5)$ well-approximating a superattracing five-cycle $(\omega_1,\dots,\omega_5).$
\paragraph{Remark.}  The random selection of $p$ is the device that breaks \alt{5}-symmetry.

\item The final piece of technology needed is the root-extractor $\Gamma_{Z_0}(w)$.   Approximate to high precision the roots $r_k$ of $R_{Z_0}$:
$$r_k=\Gamma_{Z_0}(\tilde{\omega}_k),\quad k=1,\dots,5.$$
\end{enumerate}
A \emph{Mathematica} notebook with supporting data that implements the quintic-solving procedure is available at \cite{web}.

\bibliographystyle{plain}
\bibliography{quintic_comp_ref}

\end{document}